\documentclass[11pt]{amsart}

\newtheorem{theorem}{Theorem}[section]
\theoremstyle{plain}

\newtheorem{algorithm}[theorem]{Algorithm}

\newtheorem{corollary}[theorem]{Corollary}

\newtheorem{lemma}[theorem]{Lemma}

\newtheorem{problem}[theorem]{Problem}
\newtheorem{proposition}[theorem]{Proposition}
\newtheorem{remark}[theorem]{Remark}

\numberwithin{equation}{section}

\def\aut#1{\mathrm{Aut}(#1)}
\def\mlt#1{\mathrm{Mlt}(#1)}

\def\rmlt#1{\mathrm{Mlt}_\rho(#1)}
\def\inn#1{\mathrm{Inn}(#1)}

\def\rinn#1{\mathrm{Inn}_\rho(#1)}

\def\soc#1{\mathrm{Soc}(#1)}

\begin{document}

\title[Simple automorphic loops]{Searching for small simple automorphic loops}

\author[K.~W.~Johnson]{Kenneth W.~Johnson}
\address[Johnson]{Penn State Abington, 1600 Woodland Rd, Abington, PA 19001, U.S.A.}
\email{kwj1@psu.edu}

\author[M.~K.~Kinyon]{Michael K.~Kinyon}
\address[Kinyon]{Department of Mathematics, University of Denver, 2360 S Gaylord St, Denver, Colorado 80112, U.S.A.}
\email{mkinyon@math.du.edu}

\author[G.~P.~Nagy]{G\'abor P.~Nagy}
\address[Nagy]{Bolyai Institute, University of Szeged, Aradi v\'ertan\'uk tere 1, H-6720 Sze\-ged,Hungary}
\email{nagyg@math.u-szeged.hu}

\author[P.~Vojt\v{e}chovsk\'y]{Petr Vojt\v{e}chovsk\'y}
\address[Vojt\v{e}chovsk\'y]{Department of Mathematics, University of Denver, 2360 S Gaylord St, Denver, Colorado 80112, U.S.A.}
\email{petr@math.du.edu}

\begin{abstract}
%\today.
A loop is (right) automorphic if all its (right) inner mappings are automorphisms. Using the classification of primitive groups of small degrees, we show that there is no non-associative simple commutative automorphic loop of order less than $2^{12}$, and no non-associative simple automorphic loop of order less than $2500$. We obtain numerous examples of non-associative simple right automorphic loops.

We also prove that every automorphic loop has the antiautomorphic inverse property, and that a right automorphic loop is automorphic if and only if its conjugations are automorphisms.
\end{abstract}



\thanks{G.P.~Nagy was supported by the TAMOP-4.2.1/B-09/1/KONV-2010-0005 project.}

\maketitle

\section{Introduction}

For a groupoid $Q$ and $x\in Q$, define the \emph{right translation} $R_x:Q\to Q$ by $yR_x = yx$, and the \emph{left translation} $L_x:Q\to Q$ by $yL_x = xy$. A \emph{loop} is a groupoid $Q$ with neutral element $1$ in which all translations are bijections of $Q$.

The \emph{right multiplication group} $\rmlt{Q}$ of $Q$ is the permutation group generated by all right translations of $Q$. The \emph{multiplication group} $\mlt{Q}$ of $Q$ is the permutation group generated by all translations of $Q$. The \emph{right inner mapping group} $\rinn{Q}$ of $Q$ is the stabilizer of $1$ in $\rmlt{Q}$. Equivalently, $\rinn{Q}$ is generated by all \emph{right inner mappings} $R_{x,y} = R_xR_yR_{xy}^{-1}$. The \emph{inner mapping group} $\inn{Q}$ of $Q$ is the stabilizer of $1$ in $\mlt{Q}$. Equivalently, $\inn{Q}$ is generated by all right inner mappings, all \emph{left inner mappings} $L_{x,y} = L_xL_yL_{yx}^{-1}$ and all \emph{middle inner mappings} (\emph{conjugations}) $T_x = R_xL_x^{-1}$.

Let $\aut{Q}$ be the automorphism group of a loop $Q$. Then $Q$ is a \emph{right automorphic loop} (also known as \emph{A$_r$-loop}) if $\rinn{Q}\le\aut{Q}$, and an \emph{automorphic loop} (also known as \emph{A-loop}) if $\inn{Q}\le\aut{Q}$. Note that every group is an  automorphic loop, but the converse is certainly not true.

A nonempty subset $S$ of a loop $Q$ is a \emph{subloop} of $Q$ if it is invariant under $\{R_x^\varepsilon$, $L_x^\varepsilon;\; x\in S,\,\varepsilon=\pm 1\}$. A \emph{normal subloop} of $Q$ is a subloop invariant under $\inn{Q}$, and $Q$ is \emph{simple} if it possesses no normal subloops except for the trivial subloops $Q$ and $\{1\}$.

For an introduction to the theory of loops, see \cite{Br}.

\subsection{Simple automorphic loops}

Automorphic loops were for the first time studied by Bruck and Paige \cite{BP}. The foundations of the theory of commutative automorphic loops were laid by Jedli\v{c}ka, Kinyon and Vojt\v{e}chovsk\'y in \cite{JKV1}, with such structural results such as the Cauchy Theorem, Lagrange Theorem, Odd Order Theorem, etc. A paper analogous to \cite{JKV1}, but without the assumption of commutativity, is in preparation \cite{KKPV}.

By \cite[Theorems 5.1, 5.3, 7.1]{JKV1}, every finite commutative automorphic loop is a direct product of a solvable loop of odd order and a loop of order a power of two. By \cite[Proposition 6.1 and Theorem 6.2]{JKV1}, a finite simple commutative automorphic loop is either a cyclic group of prime order, or a loop of exponent two and order a power of two. It was shown in \cite{JKV2}, by an exhaustive search with a finite model builder, that there are no simple non-associative commutative automorphic loops of order less than $32$.

By \cite{KKPV}, a non-associative finite simple automorphic loop must be of even order.

No examples of non-associative simple automorphic loops are known, and the theory of automorphic loops is not yet sufficiently developed to rule out such examples. In this paper, we use the classification of primitive groups of small degrees to show computationally:

\begin{theorem}\label{Th:MainC} There is no non-associative simple commutative automorphic loop of order less than $2^{12}$. In particular, if $Q$ is a finite commutative automorphic loop whose order is not divisible by $2^{12}$ then $Q$ is solvable.
\end{theorem}

\begin{theorem}\label{Th:MainA} There is no non-associative simple automorphic loop of order less than $2500$.
\end{theorem}

In contrast, there are some examples of non-associative finite simple right automorphic loops in the literature, mostly due to their connection to right Bruck loops and right conjugacy closed loops.

Recall that a loop is a \emph{right Bol loop} if it satisfies the identity $((zx)y)x=z((xy)x)$. The two-sided inverse $x^{-1}$ of an element $x$ is well defined in right Bol loops, and a right Bol loop is a \emph{right Bruck loop} (also known as \emph{right K-loop} or \emph{right gyrocommutative gyrogroup}) if it satisfies the identity $(xy)^{-1}=x^{-1}y^{-1}$. Funk and P.~Nagy showed, using geometric loop theory, that a right Bruck loop is a right automorphic loop \cite[Corollary 5.2]{FuNa}. (For an algebraic proof of the same result, see \cite{GoRo} or \cite{Kr}. For an introduction to Bruck loops, see \cite{Ki}.)

The first example (belonging to an infinite class of examples) of a non-associative finite simple right Bruck loop was constructed in 2007 by Nagy \cite{Na}, a loop of order $96$ and exponent $2$. The same example and another non-associative simple right Bruck loop of order $96$ (and exponent $4$) were found independently by Baumeister and Stein \cite{BS}. Both \cite{BS} and \cite{Na} built upon the work of Aschbacher \cite{As}.

A loop $Q$ is \emph{right conjugacy closed} if $R_x^{-1}R_yR_x$ is a right translation for every $x$, $y\in Q$. Every right conjugacy closed loop is right automorphic, by \cite{GoRo82}. There are unpublished and easily constructible examples of non-associative simple right conjugacy closed loops of order $8$.

As a byproduct of our search for small simple automorphic loops, we obtain a class of non-associative finite simple right automorphic loops. We know, however, that this class does not account for all non-associative finite simple right automorphic loops; notably, it does not contain any of the two simple right Bruck loops of order $96$ mentioned above.

\subsection{Open problems}

The following problems remain open:

\begin{problem}\label{Pr:C} Is there a non-associative finite simple commutative automorphic loop?
\end{problem}

Thanks to the Decomposition Theorem and Odd Order Theorem for finite commutative automorphic loops, cf. \cite{JKV1}, Problem \ref{Pr:C} has a negative answer if and only if every finite commutative automorphic loop is solvable.

\begin{problem}\label{Pr:A} Is there a non-associative finite simple automorphic loop?
\end{problem}

\subsection{Summary of content}

In \S \ref{Sc:Folders} we recall the standard construction of Baer that embeds loops into groups by means of group transversals. \S \ref{Sc:SimpleAndPrimitive} contains the well known fact that a loop is simple if and only if its multiplication group is primitive, and some information about the available libraries of primitive groups. Necessary and sufficient conditions on right translations that characterize automorphic loops are given in \S \ref{Sc:MltAndA}. These conditions are used in \S \ref{Sc:Algorithm}, where we present an algorithm that, given a transitive group $G$ on $Q$, finds all right automorphic loops $Q=(Q,*)$ such that $\rmlt{Q}\le G$ and
$G_1\le\aut{Q}$. The algorithm is further discussed in \S \ref{Sc:Search}, where we also establish Theorems \ref{Th:MainC} and \ref{Th:MainA}. Of independent interest is the fact that a right automorphic loop is automorphic if and only if all its conjugations are automorphisms, which we prove in \S \ref{Sc:Conj}. As is shown in \S \ref{Sc:Theory}, certain orders of Theorem \ref{Th:MainC} can be handled theoretically---without the algorithm of \S \ref{Sc:Algorithm}---by using known results on possible multiplication groups of loops, and from the knowledge of conjugacy classes of $G_1$, where $G$ is a primitive permutation group of affine type. We conclude the paper with a reformulation of Problem \ref{Pr:C} entirely into group theory.

\section{Loop folders in searches}\label{Sc:Folders}

It is well known since the work of Baer \cite{Ba} that every loop can be represented as a transversal in a permutation group. Since our search is based on this fact, we summarize some of his and related results here for the convenience of the reader. (See also \cite{As}.)

For a loop $Q$ on $\{1,\dots,d\}$, let $\mathcal F(Q) = (G,H,R)$ be either the triple
\begin{displaymath}
    (\rmlt{Q}, \rinn{Q}, \{R_i;\;i\in Q\}),
\end{displaymath}
or the triple
\begin{displaymath}
    (\mlt{Q},\inn{Q},\{R_i;\;i\in Q\}).
\end{displaymath}
Then $G$ is a transitive permutation group on $\{1,\dots,d\}$, $H=G_1$, and $R$ is a right transversal to $H$ in $G$, since for $g\in G$ there is a unique $i$ such that $g\in HR_i$, namely $i=1g$.

Now consider an arbitrary group $G$, $H$ a subgroup of $G$, and $R$ a right transversal to $H$ in $G$ containing $1_G$. Then we can define a binary operation  $\circ$ on $R$ by letting
\begin{displaymath}
    x\circ y = z\text{ if and only if } xy\in Hz.
\end{displaymath}
We claim that $(R,\circ)$ is a loop if and only if $R$ is a right transversal to every conjugate $H^g$ in $G$. Indeed, given $y$, $z\in R$, the equation $x\circ y = z$ has a unique solution in $R$ if and only if $x$ is the unique element of $Hzy^{-1}\cap R$; and, given $x$, $z\in R$, the equation $x\circ y=z$ has a unique solution in $R$ if and only if $y$ is the unique element of $H^x(x^{-1}z)\cap R$. There is a neutral element in $(R,\circ)$ thanks to $1_G\in R$.

Moreover, if $Q$ is a loop and $\mathcal F(Q) = (G,H,R)$, then the loop $(R,\circ)$ is isomorphic to $Q$, since $R_i\circ R_j = R_k$ if and only if $R_iR_j\in HR_k$, which happens if and only if $ij=1R_iR_j = 1HR_k = k$.

To find all loops of order $d$, it therefore suffices to consider all transitive permutation groups $G$ on $Q=\{1,\dots,d\}$, $H=G_1$, and all right transversals $R=\{r_i;\;i\in Q\}$ to all $H^g\in G$, where we can assume without loss of generality that $ir_1 = 1r_i=i$ for every $i\in Q$. Note that we can then transfer the operation $\circ$ from $R$ to the underlying set $Q$ by letting $i\circ j = k$ if and only if $r_i\circ r_j=r_k$.

It is natural to consider another operation $*$ on $Q$ by declaring the mappings $r_i$ to be the right translations of $(Q,*)$, that is, by letting $i*j = ir_j$ for $i$, $j\in Q$. Lemma \ref{Lm:SameLoops} shows that $(Q,\circ)=(Q,*)$.

\begin{lemma}\label{Lm:LoopByRightSection}
Let $R=\{r_i;\;i\in Q\}$ be a set of bijections of $Q=\{1,\dots,d\}$ such that $ir_1 = 1r_i=i$ for every $i\in Q$. Then $(Q,*)$ is a loop with neutral element $1$ if and only if $r_ir_j^{-1}$ is fixed point free for every $i\ne j\in Q$.
\end{lemma}
\begin{proof}
Note that $1$ is the neutral element of $(Q,*)$ since $i*1 = ir_1 = i = 1r_i = 1*i$ for every $i\in Q$. By definition, the right translation $R_j$ by $j$ in $(Q,*)$ coincides with $r_j$. Moreover, the following conditions are equivalent for $i$, $j$, $k\in Q$: $kr_ir_j^{-1} = k$, $kr_i = kr_j$, $k*i = k*j$.

If $(Q,*)$ is a loop, we deduce that $r_ir_j^{-1}$ is fixed point free whenever $i\ne j$. Conversely, if $r_ir_j^{-1}$ is fixed point free for every $i\ne j$, we see that every left translation $L_k$ in $(Q,*)$ is one-to-one, hence onto. Since, by assumption, every right translation of $(Q,*)$ is a bijection, $(Q,*)$ is a loop.
\end{proof}

\begin{lemma}\label{Lm:SameLoops}
Let $G$ be a transitive permutation group on $Q=\{1,\dots,d\}$, $H=G_1$, and $R=\{r_i;\;i\in Q\}\subseteq G$ such that $ir_1 = 1r_i=i$ for every $i\in Q$. Then $R$ is a transversal to every conjugate $H^g$ in $G$ if and only if $r_ir_j^{-1}$ is fixed point free for every $i\ne j\in Q$. If this condition is satisfied, the loops $(Q,\circ)$ and $(Q,*)$ coincide.
\end{lemma}
\begin{proof}
Assume that $R$ is a transversal to every conjugate $H^g$ in $G$. Then $(Q,\circ)$ is a loop, where, recall, $i\circ j=k$ if and only if $r_ir_j\in Hr_k$. Moreover, $1r_ir_j = ir_j = 1Hr_{ir_j}$, so $r_ir_j\in Hr_{ir_j}$, $i\circ j = ir_j = i*j$, and $(Q,*)=(Q,\circ)$ coincide.

Conversely, assume that $r_ir_j^{-1}$ is fixed point free for every $i\ne j$. Then $(Q,*)$ is a loop with neutral element $1$ by Lemma \ref{Lm:LoopByRightSection}. Since $R$ is a right transversal to $H$ in $G$, $(Q,\circ)$ is defined. The equality $i\circ j = i*j$ then follows as above.
\end{proof}

Constructing all loops from suitable subsets $R$ of right translations in transitive permutation groups is obviously prohibitive already for rather small values of $d$. But we can take advantage of the following results that greatly restrict the possible transitive groups $G$ in general, and the subsets $R$ in the case of automorphic loops.

\section{Simple loops and primitive groups}\label{Sc:SimpleAndPrimitive}

Recall that a transitive permutation group on $Q$ is \emph{primitive} if it preserves no nontrivial partition of $Q$. The \emph{degree} of a primitive group is the number of points it moves, that is, the cardinality of $Q$. Note that every $2$-transitive group is primitive.

The following result is well known, with earliest reference likely \cite[Theorem 8]{Al}:

\begin{proposition}[Albert]\label{Pr:Simple} A loop $Q$ is simple if and only if its multiplication group $\mlt{Q}$ is primitive on $Q$.
\end{proposition}
\begin{proof}
Assume that $Q$ is not simple, and let $S$ be a nontrivial normal subloop of $Q$. Then for every $x$, $y\in Q$ we have $xS=Sx$ since $S$ is invariant under conjugations, $(yS)L_x = x(yS) = (xy)S$ since $S$ in invariant under left inner mappings, and $(yS)R_x = (Sy)R_x = (Sy)x = S(yx) = (yx)S$ since $S$ is invariant under right inner mappings. Thus $\mlt{Q}$ preserves the nontrivial partition $\{yS;\;y\in Q\}$ of $Q$, so it is not primitive on $Q$.

Conversely, assume that $\mlt{Q}$ is not primitive on $Q$, and let $\{B_1,\dots,B_m\}$ be a nontrivial partition of $Q$ preserved by $\mlt{Q}$. Without loss of generality, let $S=B_1$ be the block containing $1$. With $x$, $y\in Q$, both $SL_x$ and $SR_x$ contain $x$, so $xS=Sx$, and, similarly, $x(yS) = (xy)S$ and $(Sx)y = S(xy)$. Hence $S$ is a normal subloop of $Q$. If $|S|=1$ then $|B_j|=1$ for every $j$, as $\mlt{Q}$ acts transitively on $Q$, a contradiction.
\end{proof}

Building upon the work of O'Nan and Scott, Aschbacher, Dixon and Mortimer, to name a few, Roney-Dougal classified all primite groups of degree less than $2500$ \cite{Ro}. (See \cite[Section 1]{Ro} for an extensive historical background concerning the classification.)

These groups are conveniently accessed in the GAP \cite{GAP} library ``Primitive Permutation Groups''. The GAP command \texttt{NrPrimitiveGroups(}$d$\texttt{)} returns the number of primitive groups of degree $d$, and the $i$th primitive group of degree $d$ is retrieved with the command \texttt{PrimitiveGroup(}$d,i$\texttt{)}.

The main reason why we were not able to expand the scope of Theorems \ref{Th:MainC} and \ref{Th:MainA} is the extent of the available libraries of primitive groups.

\section{Multiplication groups of automorphic loops}\label{Sc:MltAndA}

It follows from the classification of finite simple groups that the only $4$-transitive groups are the symmetric groups $S_n$ for $n\ge 4$, the alternating groups $A_n$ for $n\ge 6$, and the Mathieu groups $M_{11}$, $M_{12}$, $M_{23}$ and $M_{24}$. Corollary \ref{Cr:4Transitive} below therefore does not disqualify many primitive groups from being multiplication groups of automorphic loops, but, importantly, it disqualifies the computationally most difficult symmetric and alternating groups.

\begin{remark}
It appears that it is rare for a simple loop to have a multiplication group different from $A_n$ and $S_n$. This statement could likely be made more precise by modifying Cameron's proof \cite{Ca} of the following result: \emph{The rows (viewed as permutations) of a randomly chosen latin square of order $n$ generate either $A_n$ or $S_n$ with probability approaching $1$ as $n$ approaches infinity.}
\end{remark}

\begin{lemma}\label{Lm:4Transitive} Let $Q$ be a loop and $H$ a subgroup of $\aut{Q}$. Then for every $i$, $j\in Q$ the product $ij$ belongs to a trivial orbit of the pointwise stabilizer $H_{i,j}$. In particular, $H$ is not $3$-transitive on $Q\setminus\{1\}$, except for the case $Q=C_2\times C_2$ and $H=\aut{Q}=S_3$.
\end{lemma}
\begin{proof}
Assume that $ij=k$ and $k$ is not in a trivial orbit of $H_{i,j}$. Then there is $h\in H$ such that $ih=i$, $jh=j$ and $kh\ne k$. Thus $ij = ih\cdot jh = (ij)h = kh\ne k$, a contradiction.

Assume that $H$ is $3$-transitive on $Q\setminus\{1\}$. If $|Q|>4$, then the transitive $H_{i,j}$ has a nontrivial orbit. The only loops of order $4$ are $C_4$ with $\aut{C_4}=C_2$, and $C_2\times C_2$ with $\aut{C_2\times C_2}=S_3$.
\end{proof}

\begin{corollary}\label{Cr:4Transitive}
Let $Q$ be an automorphic loop. Then $\inn{Q}$ is not $3$-transitive on $Q\setminus\{1\}$ and $\mlt{Q}$ is not $4$-transitive on $Q$.
\end{corollary}
\begin{proof}
Since $\inn{Q}=\mlt{Q}_1\le\aut{Q}$, we are done by Lemma \ref{Lm:4Transitive} as long as $Q\ne C_2\times C_2$. But if $Q=C_2\times C_2$ then $\inn{Q}=1$.
\end{proof}

The right translations of an automorphic loop are linked by the action of the inner mapping group:

\begin{lemma}\label{Lm:ACharacterization}
Let $Q$ be a loop and $h$ a permutation of $Q$. Then $h\in\aut{Q}$ if and only if $R_i^h = R_{ih}$ for every $i\in Q$.
\end{lemma}
\begin{proof}
The following conditions, all universally quantified for $j\in Q$, are equivalent for $i$ and $h$: $R_i^h = R_{ih}$, $jh^{-1}R_ih = jR_{ih}$, $(jh^{-1}\cdot i)h = j(ih)$, $(ji)h = jh\cdot ih$.
\end{proof}

\begin{corollary}\label{Cr:ALoop}
A loop $Q$ is automorphic if and only if $R_i^h = R_{ih}$ for every $i\in Q$ and $h\in \inn{Q}$.
\end{corollary}

Here is a summary of results that will be used to explain the algorithm of \S \ref{Sc:Algorithm}:

\begin{proposition}\label{Pr:ForASearch}
Let $Q$ be a loop and $H\le\aut{Q}$. Then
\begin{enumerate}
\item[(i)] $R_i^h = R_{ih}$ for every $i\in Q$ and $h\in H$,
\item[(ii)] $|R_i^H| = |iH|$ for every $i\in Q$,
\item[(iii)] there exists $I\subseteq Q$ such that $\sum_{i\in I} |iH| = |Q|$ and such that
$\{R_i;\;i\in Q\}$ is the disjoint union $\bigcup_{i\in I} R_i^H$,
\item[(iv)] $R_i$ commutes with every element of the stabilizer $H_i$, for $i\in Q$,
\item[(v)] $R_iR_j^{-1}$ is fixed point free for every distinct $i$, $j\in Q$.
\end{enumerate}
\end{proposition}
\begin{proof}
Parts (i), (ii) and (iii) follow from Lemma \ref{Lm:ACharacterization}. Let $h\in H_i$. Then for every $j\in Q$ we have $jhR_i = jh\cdot i = jh\cdot ih = (ji)h = jR_ih$, proving (iv). Part (v) follows from Lemma \ref{Lm:LoopByRightSection}.
\end{proof}

\section{Loops with prescribed automorphisms}\label{Sc:Algorithm}

Let $G$ be a transitive permutation group on a finite set $Q$, and let $H=G_1$. The following algorithm efficiently searches for all loops $Q=(Q,*)$ (with fixed neutral element $1$) such that $\rmlt{Q}\le G$ and $H\le\aut{Q}$. The loops $(Q,*)$ will be constructed by means of the set $R =\{r_i;\;i\in Q\}\subseteq G$, where $r_i$ will be the right translation by $i$ in $(Q,*)$. As usual, we can assume without loss of generality that $ir_1 = 1r_i=i$ for every $i\in Q$. All references in the algorithm are to Proposition \ref{Pr:ForASearch}:

\bigskip
\hrule
\bigskip

\begin{algorithm}\label{Al:Main}
\end{algorithm}

\emph{Step 1}: Set $r_1=1_G$. Find $I\subseteq Q$ such that the disjoint union $\bigcup_{i\in I} iH$ is equal to $Q\setminus\{1\}$.

\emph{Step 2}: For $i\in I$, find $\mathcal R_i$, the set consisting of all candidates $r_i$ for the right translation by $i$, as follows: By (iv), (v) and the fact that $r_1=1_G\in R$, $r_i$ must be fixed point free, in the centralizer $C_G(H_i)$, and such that $1r_i=i$. If there is no such $r_i$, the algorithm stops with failure. Else it suffices to find one such $r_i$, and set
$\mathcal R_i$ equal to the coset $(C_G(H_i))_1r_i$, because $1s=i$ if and only if $1sr_i^{-1}=1$.

\emph{Step 3}: For $i\in I$, find $\mathbf R_i$, the set consisting of candidate orbits $r_i^H$ with $r_i\in\mathcal R_i$, as follows: Let $r_i\in\mathcal R_i$. If $|r_i^H|\ne |iH|$, discard $r_i$, by (ii). If there is $s\in r_i^H$ such that $s\ne r_i$ and $sr_i^{-1}$ is not fixed point free, discard $r_i$, by (v). Else add $r_i^H$ into $\mathbf R_i$.

\emph{Step 4}: For $i$, $j\in I$, decide which pairs $r_i^H\in \mathbf R_i$, $r_j^H\in \mathbf R_j$ do not contradict (v): Call two candidate orbits $r_i^H$, $r_j^H$ with $i\ne j$ \emph{compatible} if $st^{-1}$ is fixed point free for every $s\in r_i^H$ and $t\in r_j^H$. Compatibility is a symmetric relation, so it suffices to consider orbits $r_i^H$, $r_j^H$ with $i<j$. To decide if $r_i^H$, $r_j^H$ with $i<j$ are compatible, it suffices to check that all permutations in $r_j^Hr_i^{-1}$ (rather than in $r_j^H(r_i^H)^{-1}$) are fixed point free. Indeed, if $kr_i^{h_i} = kr_j^{h_j}$ for some $k\in Q$ and $h_i$, $h_j\in H$, then $(kh_i^{-1})r_i = (kh_i^{-1})r_j^{h_jh_i^{-1}}$.

\emph{Step 5}: Put together pairwise compatible candidate orbits to form the set of loop translations $\{r_i;\;1<i\in Q\}$: This can be done elegantly with the use of graph algorithms in the GAP package GRAPE \cite{GRAPE}. The compatibility relation from Step 4 corresponds to the edges of a graph $\mathcal G$ whose vertices are the candidate orbits $r_i^H$. Assign vertex weight $|r_i^H|$ to $r_i^H$, and return all complete subgraphs of $\mathcal G$ whose vertex weights add up to $|Q|-1$.

\bigskip
\hrule
\bigskip

Here is the GAP code for the algorithm (it can be downloaded from the web site of the last author, \texttt{http://www.math.du.edu/\~{}petr}):

\begin{small}
\begin{verbatim}
RightAutomorphicLoopsWithPrescribedAutomorphisms := function(g)
# returns all loops Q whose right multiplication group is a subgroup of g
# and Stabilizer(g, 1) is a subgroup of Aut(Q)
    local d, h, c, v, ls, x, i, j, k, orbs, orbit, new, graph, comp, cs;
    # Step 1
    d := NrMovedPoints(g);
    h := Stabilizer(g, 1);
    orbs := Set(Orbits(h, [2..d]), Set);
    # Steps 2 and 3
    ls := [];
    for orbit in orbs do
        c := Centralizer(g, Stabilizer(h, orbit[1]));
        v := RepresentativeAction(c, 1, orbit[1]);
        if v <> fail then
            v := RightCoset(Stabilizer(c, 1), v);
            v := Filtered(v, x -> NrMovedPoints(x) = d);
            new := [];
            for x in v do
                k := Orbit(h, x);
                if ForAll(k/k[1], y -> NrMovedPoints(y) in [0,d])
                    and Length(k) = Length(orbit)
                    then Add(new, k);
                fi;
            od;
            Add(ls, new);
        fi;
    od;
    ls := Concatenation(ls);
    # Step 4
    comp := List([1..Length(ls)], i -> []);
    for i in [1..Length(ls)] do
        for j in [1..i] do
            if ForAll(ls[i]/ls[j][1], p -> NrMovedPoints(p) = d) then
                comp[i][j] := true;
                comp[j][i] := true;
            else
                comp[i][j] := false;
                comp[j][i] := false;
            fi;
        od;
    od;
    # Step 5
    graph := Graph(Group(()), [1..Length(comp)], OnPoints,
        function(x, y) return comp[x][y]; end);
    cs := CompleteSubgraphsOfGivenSize(
            graph, d-1, 1, false, true, List(ls, Length));
    cs := List(cs, x -> VertexNames(graph){x});
    cs := List(cs, x -> Concatenation(ls{x}));
    cs := List(cs, x -> SortedList(Concatenation([()], x)));
    return cs;
end;
\end{verbatim}
\end{small}

\section{The search for simple (right) automorphic loops}\label{Sc:Search}

Since the algorithm of \S \ref{Sc:Algorithm} is delicate, we first present some comments and then give the results. We assume that the input of the algorithm is a permutation group $G$ primitive on the set $Q$.

\subsection{Discussion of the algorithm}

\emph{The algorithm returns all loops $Q=(Q,*)$ such that $\rmlt{Q}\le G$ and $G_1=H\le\aut{Q}$.} Indeed, the inclusion $\rmlt{Q}\le G$ is obvious. Step 3 and Lemma \ref{Lm:ACharacterization} guarantee that $H\le\aut{Q}$, since for every $h\in H$ and $i\in Q$ we have $1r_i^h = 1h^{-1}r_ih = 1r_ih = ih$, hence $r_{ih} = r_i^h$. Steps 3, 4 and 5 guarantee that every $r_ir_j^{-1}$ with $i\ne j$ is fixed point free, so $(Q,*)$ is a loop by Lemma \ref{Lm:LoopByRightSection}.

\emph{All returned loops are right automorphic.} We have $\rmlt{Q}\le G$, so $\rinn{Q}\le G_1\le\aut{Q}$.

\emph{Not all returned loops are necessarily simple.} The condition $\rmlt{Q}\le G$ does not guarantee that either $\rmlt{Q}$ or $\mlt{Q}$ is primitive. If it happens that $\rmlt{Q}=G$ then both $\rmlt{Q}$ and $\mlt{Q}$ are primitive, hence $Q$ is simple.

\emph{Not all finite simple right automorphic loops are found.} Let $Q$ be a simple right automorphic loop and $G=\mlt{Q}$. Then the algorithm with input $G$ returns $Q$ if and only if $\inn{Q}=G_1\le\aut{Q}$, that is, if and only if $Q$ is automorphic. Thus, when $Q$ is right automorphic but not automorphic, it will not be found with input $G$, but it \emph{could} be found with a different primitive group as the input. Furthermore, if $Q$ is a simple right automorphic that is not automorphic and if $\rmlt{Q}$ is imprimitive (such loops exist), then $Q$ will not be found by the algorithm applied to any primite group.

\emph{We can skip $4$-transitive groups.} If $G$ is $4$-transitive then $H=G_1$ is $3$-transitive and no non-associative loop $Q$ with $H\le\aut{Q}$ exists, by Lemma \ref{Lm:4Transitive}.

\emph{While searching for simple automorphic loops, it suffices to consider groups of even degree.} By a result of \cite{KKPV}, a non-associative finite simple automorphic loop is of even order.

\emph{While searching for simple commutative automorphic loops, it suffices to consider groups of degree a power of two.} By \cite[Proposition 6.1 and Theorem 6.2]{JKV1}, a non-associative finite simple commutative automorphic loop is of order a power of two.

\emph{While searching for simple automorphic loops, we can skip solvable groups.} Vesanen proved \cite{VeJA} that any loop with solvable multiplication group is itself solvable. Hence if $Q$ is a non-associative simple automorphic loop then $G=\mlt{Q}$ is not solvable, and $Q$ will be found by the algorithm with input $G$. (Of course, $Q$ could also be found by the algorithm with some solvable group as the input.)

\emph{While systematically searching for simple automorphic loops, it is not necessary to check that all inner mappings are automorphisms.} If $\mlt{Q}=G$ then $\inn{Q}\le\aut{Q}$ and $Q$ is a simple automorphic loop. If $\mlt{Q}\ne G$ then we can ignore $Q$ because either $\mlt{Q}$ is not primitive (and $Q$ is not simple), or $\mlt{Q}$ is primitive, in which case $Q$ will be found again by the algorithm with input $\mlt{Q}$.

\subsection{Results}

\emph{Simple right automorphic loops.}

We found all non-associative simple right automorphic loops $Q$ up to isomorphism with the following properties: $|Q|<504$, there exists a primitive group $G$ of degree $|Q|$ such that $\rmlt{Q}\le G$ and $G_1\le\aut{Q}$.

The following table summarizes the results:
\begin{displaymath}
\begin{array}{c||c|c|c|c|c|c|c|c|c|c|c}
\text{order}&15&27&60&64&81&125&168&243&256&343&360\\
\hline
\text{found loops}&1&1&5&1&2&6&11&60&2&28&17
\end{array}\quad.
\end{displaymath}

To do this, it suffices to (i) apply Algorithm \ref{Al:Main} to all primitive groups $G$ of degree less than $504$ that are not $4$-transitive, (ii) to filter the resulting loops for simplicity, (iii) to filter all remaining loops up to isomorphism.

Concerning (ii): In most cases we quickly observe $\rmlt{Q}=G$, which means that $G$ is simple. In the few remaining cases when $\rmlt{Q}<G$ we must calculate $\mlt{Q}$ and check for primitivity.

Concerning (iii): We used the isomorphism filter for loops built into the Loops \cite{LOOPS} package of GAP. The isomorphism filtering takes up most of the running time of the search, and it is one of the reasons why we decided to stop at order $504$. The other reason is that the search does not find all non-associative simple right automorphic loops, as we have already pointed out in the Introduction, so it is not clear how useful the results are for large orders.

The running time of the search was about 30 minutes on a 2GHz processor PC.

\emph{Simple automorphic loops.}

By running the algorithm on all primitive groups of even degree less than $2500$ that are neither $4$-transitive nor solvable, we established Theorem \ref{Th:MainA}, a somewhat surprising result.

The running time of the search was about 20 minutes.

\emph{Simple commutative automorphic loops.}

Theorem \ref{Th:MainC} now follows, too. But it is possible to obtain it faster, by running the algorithm on all primitive groups of degree a power of two and less than $2^{12}$ that are neither $4$-transitive nor solvable.

The running time of the search was about $2$ minutes.

\section{Right automorphic loops and conjugations}\label{Sc:Conj}

As discussed in \S \ref{Sc:Search}, it is never necessary to check that inner mappings are automorphisms while systematically searching for simple automorphic loops by Algorithm \ref{Al:Main}. But it is necessary to run the check if we wish to find all automorphic loops with $\rmlt{Q}\le G$ and $G_1\le \aut{Q}$ for a fixed primitive group $G$. The following result, which is of independent interest, shows that it is not necessary to check the left inner mappings.

\begin{theorem}\label{Th:RightPlusConjugations}
Let $Q$ be a right automorphic loop. Then $Q$ is automorphic if and only if all conjugations $T_x$, $x\in Q$, are automorphisms of $Q$.
\end{theorem}

The rest of this section forms the proof of Theorem \ref{Th:RightPlusConjugations}.

Recall that a loop is \emph{flexible} if it satisfies the identity $xy\cdot x = x\cdot yx$, that is, $L_x R_x = R_x L_x$ for all $x$. Flexible loops have two-sided inverses. Indeed, if $x^{\lambda}$ and $x^{\rho}$ denote the left and right inverses of $x$, respectively, then $x=1\cdot x = xx^\rho\cdot x = x\cdot x^\rho x$ implies $x^\rho x = 1$, and since $x^\lambda x=1$, we have $x^\rho = x^\lambda = x^{-1}$.

\begin{lemma}
\label{lem:flex}
Let $Q$ be a loop in which every conjugation $T_x$, $x\in Q$, is an automorphism.
Then $Q$ is flexible.
\end{lemma}
\begin{proof}
For $x$, $y\in Q$ we have $yL_xT_x = (xy)T_x = xT_x\cdot yT_x = x\cdot yT_x = yT_xL_x$ because $T_x$ is an automorphism. Thus $L_xT_x = T_xL_x$ and $L_x R_x = L_x T_x L_x = T_x L_x L_x = R_x L_x$.
\end{proof}

\begin{lemma}
\label{lem:switch}
Let $Q$ be a flexible, right automorphic loop. Then for all
$x\in Q$,
\begin{align}
R_{x,x^{-1}} &= R_{x^{-1},x}\,, \label{eq:switch1} \\
L_x R_{x^{-1}} &= R_{x^{-1}} L_x \,, \label{eq:switch2} \\
L_{x^{-1}} R_x^{-1} &= R_x^{-1} L_{x^{-1}} \,. \label{eq:switch3}
\end{align}
\end{lemma}

\begin{proof}
Note that $x^{-1}R_{x,x^{-1}} = x^{-1}$, thus $yR_{x^{-1}}R_{x,x^{-1}} = (yx^{-1})R_{x,x^{-1}} =
yR_{x,x^{-1}}\cdot x^{-1}R_{x,x^{-1}} = yR_{x,x^{-1}}\cdot x^{-1} = yR_{x,x^{-1}}R_{x^{-1}}$, or $R_{x^{-1}}R_{x,x^{-1}} = R_{x,x^{-1}}R_{x^{-1}}$. Then $R_{x^{-1},x} R_{x^{-1}} = R_{x^{-1}} R_x R_{x^{-1}} = R_{x^{-1}} R_{x,x^{-1}} = R_{x,x^{-1}} R_{x^{-1}}$,
which yields \eqref{eq:switch1}.

Similarly, by \eqref{eq:switch1}, we have $x R_{x,x^{-1}} = x R_{x^{-1},x} = x$, therefore $yL_xR_{x,x^{-1}} = (xy)R_{x,x^{-1}} = xR_{x,x^{-1}}\cdot yR_{x,x^{-1}} = x\cdot yR_{x,x^{-1}} = yR_{x,x^{-1}}L_x$, or $L_xR_{x,x^{-1}} = R_{x,x^{-1}}L_x$. Then, by flexibility, $R_x L_x R_{x^{-1}} = L_x R_x R_{x^{-1}} = L_x R_{x,x^{-1}} = R_{x,x^{-1}}L_x = R_xR_{x^{-1}}L_x$, and \eqref{eq:switch2} follows.

Finally, \eqref{eq:switch3} follows from \eqref{eq:switch2} upon replacing $x$ with $x^{-1}$ and rearranging.
\end{proof}

A loop with two-sided inverses is said to have the \emph{antiautomorphic inverse property} if it satisfies the identity
$(xy)^{-1} = y^{-1} x^{-1}$. If we use $J$ to denote the inversion permutation $x\mapsto x^{-1}$, then the antiautomorphic inverse property is equivalent to $R_y^J = L_{y^{-1}}$ for all $y$, or to $L_y^J = R_{y^{-1}}$ for all $y$.

\begin{proposition}
\label{Pr:AAIP}
A flexible, right automorphic loop has the antiautomorphic inverse property.
\end{proposition}

\begin{proof}
By \eqref{eq:switch3}, $(x^{-1} y^{-1}) R_x^{-1} = y^{-1} R_x^{-1} L_{x^{-1}}$. Let us apply $R_{x,y}$ to both sides. On the left side we get $(x^{-1} y^{-1}) R_x^{-1} R_{x,y} = (x^{-1} y^{-1}) R_y R_{xy}^{-1}$. On the right side we get $y^{-1}R_x^{-1}L_x^{-1}R_{x,y} = (x^{-1}\cdot y^{-1}R_x^{-1})R_{x,y} = x^{-1}R_{x,y}\cdot y^{-1}R_x^{-1}R_{x,y} = yR_{xy}^{-1}\cdot (xy)^{-1} = yR_{xy}^{-1}R_{(xy)^{-1}}$.

Hence $(x^{-1} y^{-1}) R_y R_{xy}^{-1} = y R_{xy}^{-1} R_{(xy)^{-1}}$, so $(x^{-1} y^{-1}) R_y = y R_{xy}^{-1} R_{(xy)^{-1}} R_{xy} =
y R_{xy}^{-1} R_{xy} R_{(xy)^{-1}} = y R_{(xy)^{-1}} = x R_y J L_y$, where we have used \eqref{eq:switch1} in the second equality.
Now, $(x^{-1} y^{-1}) R_y = x J R_{y^{-1},y} = x R_{y^{-1},y} J = x R_{y,y^{-1}} J = x R_y R_{y^{-1}} J$, using the fact that $R_{y^{-1},y}$ is an automorphism in the second equality, and \eqref{eq:switch1} in the third.

Thus we have $R_y J L_y = R_y R_{y^{-1}} J$, or $L_y^J = R_{y^{-1}}$, which is the antiautomorphic inverse property.
\end{proof}

Combining Lemma \ref{lem:flex} and Proposition \ref{Pr:AAIP} yields:

\begin{theorem}
Every automorphic loop has the antiautomorphic inverse property.
\end{theorem}

We can now finish the proof of Theorem \ref{Th:RightPlusConjugations} as follows: The necessity of the condition is obvious, so let us prove sufficiency. By Lemma \ref{lem:flex}, $Q$ is flexible. By Proposition \ref{Pr:AAIP}, $Q$ has the antiautomorphic inverse property. For each $x$, $y\in Q$, $R_{x,y}^J = JR_{x,y}J = R_{x,y}$, since $R_{x,y}$ is an automorphism. Hence $R_{x,y} = R_{x,y}^J = R_x^J R_y^J (R_{xy}^{-1})^J =  R_x^J R_y^J (R_{xy}^J)^{-1} = L_{x^{-1}} L_{y^{-1}} L_{(xy)^{-1}}^{-1} = L_{x^{-1}} L_{y^{-1}} L_{y^{-1}x^{-1}}^{-1} = L_{x^{-1},y^{-1}}$ by the antiautomorphic inverse property. This implies that every inner mapping of $Q$ is an automorphism, and we are through.

\section{Additional theoretical results}\label{Sc:Theory}

We have by now proved the theorems from the Introduction by a computer search based on Algorithm \ref{Al:Main} and several theoretical results. However, as we are going to explain next, some degrees $d=2^m$ can be eliminated without such a search, by taking advantage of certain results about primitive groups of affine type and permutation groups realizable as multiplication groups of loops. We will still need some computer calculations, but only of rather routine character, such as determining the size of the smallest nontrivial conjugacy class of a given group.

\subsection{Groups that are (not) multiplication groups of loops}

The question of which (transitive permutation) groups are multiplication groups of loops has been studied but remains largely unanswered. Although we will only need some of the known results below, we include them for the sake of completeness. All groups are assumed to be in their natural permutation representation.

\begin{proposition}\label{Pr:Are} The following groups are multiplication groups of loops:
\begin{enumerate}
\item[(i)] $S_n$ for $n\ge 2$ \cite{DrMS},
\item[(ii)] $A_n$ for $n\ge 6$ \cite{DrKe},
\item[(iii)] $M_{12}$ \cite{Co},
\item[(iv)] $M_{24}$ \cite{NaEUJC}.
\end{enumerate}
If $n\ge 3$ and $q^n>8$ then there is a loop $Q$ with $PSL(n,q)\le\mlt{Q}\le PGL(n,q)$ \cite{NaEUJC}.
\end{proposition}

\begin{proposition}\label{Pr:AreNot} The following groups are not multiplication groups of loops:
\begin{enumerate}
\item[(i)] $PSL(2,q)$ for $q\ge 3$ \cite{VeCA},
\item[(ii)] $M_{11}$, $M_{23}$ \cite{DrDim2},
\item[(iii)] $PSL(2n,q)$ with $q\ge 5$, $PU(n,q^2)$ with $n\ge 6$, $PO(n,q)$ with $n\ge 7$ odd, $PO^\varepsilon(n,q)$ with $n\ge 7-\varepsilon$ even \cite{VeCamb}.
\end{enumerate}
Furthermore, if every $1\ne \alpha\in\mlt{Q}$ fixes at most two points then $Q$ is an abelian group \cite{Dr2Points}.
If $\mlt{Q}\le P\Gamma L(2,q)$ and $q\ge 5$ then $Q$ is an abelian group \cite{DrDim2}.
\end{proposition}

\subsection{Primitive groups of affine type}

Recall that the \emph{socle} $\soc{G}$ of a group $G$ is the subgroup (necessarily normal) of $G$ generated by all minimal normal subgroups of $G$. Also recall that a permutation group $G$ on $X$ is \emph{regular} if it is sharply transitive, that is, for every $i$, $j\in X$ there is unique $g\in G$ such that $ig=j$.

Of great importance in the classification of primitive groups is the O'Nan-Scott Theorem (see, for instance, \cite{Ro}). We will only need the part of the O'Nan-Scott Theorem concerned with abelian socle:

\begin{theorem}
Let $G$ be a primitive group of degree $d$, and let $U=\soc{G}$. Then $U$ is abelian if and only if $U$ is regular, elementary abelian $p$-group of order $d=p^n$ and $G$ is isomorphic to a subgroup of the affine linear group $AGL(n,p)$.
\end{theorem}

Primitive groups with abelian socle are therefore called of \emph{affine type}.

Let $Q$ be a simple commutative automorphic loop, and let $R=\{r_i;\;i\in Q\}$ be the right translations of $Q$, with $ir_1=1r_i=i$ for $i\in Q$. Assume that $G=\mlt{Q}$ is a primitive group of affine type of degree $d$ and $H=G_1$. Let also $U=\soc{G}$. Since $U$ is a normal regular subgroup of $G$, we have $|U|=|Q|=d$, and we can write $U=\{u_i;\;i\in Q\}$ with some $u_i$ satisfying $1u_i=i$. Note that $U$ is a right transversal to $H$ in $G$. There are thus uniquely determined $h_i\in H$ such that $r_i = h_iu_i$ for $i\in Q$. We will call this situation the \emph{affine setup}.

Note that in the affine setup we can define an isomorphic copy of the loop $Q$ on $U$ by letting $u_i\bullet u_j = u_i^{h_j}u_j$, since $r_ir_j = h_iu_ih_ju_j = h_ih_ju_i^{h_j}u_j \in Hr_{ij} = Hh_{ij}u_{ij} = Hu_{ij}$.

\begin{lemma}\label{Lm:Generators} In the affine setup, $\langle h_i;\;i\in Q\rangle = H$.
\end{lemma}
\begin{proof}
Note that $r_ir_jr_{ij}^{-1}r_{ij} = r_ir_j = h_ih_ju_i^{h_j}u_j = h_ih_ju_{ij} = h_ih_jh_{ij}^{-1}r_{ij}$, and thus $r_ir_jr_{ij}^{-1}= h_ih_jh_{ij}^{-1}$. Since $Q$ is commutative, we are done by $H=\inn{Q} = \rinn{Q} = \langle r_ir_jr_{ij}^{-1};\;i,j\in Q\rangle = \langle h_ih_jh_{ij}^{-1};\;i,j\in Q\rangle \le \langle h_i;\;i\in Q\rangle\le H$.
\end{proof}

We will need the following result of Niemenmaa and Kepka \cite{NiKe1994}:

\begin{theorem}[Niemenmaa and Kepka]\label{Th:NiemenmaaKepka}
Let $Q$ be a finite loop such that $\inn{Q}$ is abelian. Then $Q$ is nilpotent.
\end{theorem}

\begin{remark} Building upon the work of Mazur \cite{Ma}, Niemenmaa obtained a more general result in \cite{Ni}: \emph{Let $Q$ be a finite loop such that $\inn{Q}$ is nilpotent. Then $Q$ is nilpotent.}
\end{remark}

\begin{proposition}\label{Pr:InCenter}
In the affine setup, if $h_i\in Z(H)$ for every $i\in Q$, then $Q$ is a cyclic group of prime order.
\end{proposition}
\begin{proof}
Assume that $h_i\in Z(H)$ for every $i\in Q$. By Lemma \ref{Lm:Generators}, $\inn{Q}=H=Z(H)$ is an abelian group. By Theorem \ref{Th:NiemenmaaKepka}, $Q$ is nilpotent. Since $Q$ is simple, it follows that $Q=Z(Q)$ is a simple abelian group, necessarily a cyclic group of prime order.
\end{proof}

\begin{proposition}\label{Pr:ConjugacyClasses} In the affine setup, let $\gamma$ be the size of a largest orbit of $H$ on $Q$. If $Q$ is not associative, then $H$ contains a conjugacy class $C$ of size $1<|C|\le \gamma$.
\end{proposition}
\begin{proof}
If every conjugacy class $h_i^H$ is trivial then $Q$ is a group by Proposition \ref{Pr:InCenter}. Thus we can assume that there is $C=h_i^H$ such that $|C|>1$.
By Lemma \ref{Lm:ACharacterization}, $h_i^mu_i^m = r_i^m = r_{im} = h_{im}u_{im}$ for every $m\in H$. Since $U$ is normal in $G$, $u_i^m = u_j$ for some $j$, in fact, $j = 1u_j = 1u_i^m = 1m^{-1}u_im = 1u_im = im$. Thus $u_i^m = u_{im}$, and $h_i^m = h_{im}$ follows. Then
$C = h_i^H = \{h_{im};\;m\in H\}$, and thus $|C|$ cannot exceed the size of the $H$-orbit of $i$ on $Q$. In particular, $|C|\le \gamma$.
\end{proof}

\subsection{Simple commutative automorphic loops of orders $32$ and $128$}

To illustrate the theoretical results, we show, without the search of \S \ref{Sc:Algorithm}, that there are no non-associative simple commutative automorphic loops of orders $32$ and $128$.

\textit{Order $32$:} The groups $A_{32}$ and $S_{32}$ are $4$-transitive, and can be eliminated by Lemma \ref{Lm:4Transitive}. The groups $AGL(1,32)$ and $A\Gamma L(1,32)$ are solvable, eliminated by \cite{VeJA}. The groups $PSL(2,31)$ and $PGL(2,31)$ are eliminated by Proposition \ref{Pr:AreNot}. The group $ASL(5,2)$ has no nontrivial conjugacy class of size less than $32$, so it is eliminated by Proposition \ref{Pr:ConjugacyClasses}. There are no other primitive groups of degree $32$.

\textit{Order $128$:} The groups $A_{128}$, $S_{128}$ are $4$-transitive, and the groups $AGL(1,128)$, $A\Gamma L(1,128)$ are solvable. The groups $PSL(2,127)$, $PGL(2,127)$ are eliminated by Proposition \ref{Pr:AreNot}. The group $AGL(2,7)$ has no nontrivial conjugacy class of size less than $128$, so it is eliminated by Proposition \ref{Pr:ConjugacyClasses}. There are no other primitive groups of degree $128$.

\section{Reformulation of the main problem to group theory}\label{Sc:Reformulation}

We conclude this paper by restating Problem \ref{Pr:C} entirely within group theory. We claim that Problem \ref{Pr:C} is equivalent to the following:

\begin{problem}\label{Pr:CG} Is there a set $Q$ containing $1$, a permutation group $G$ on $Q$, and a subset $R\subseteq G$ containing $1_G$ such that:
\begin{enumerate}
\item[(a)] $G$ is primitive on $Q$ and $|Q|=2^n > 2$,
\item[(b)] $R$ is a right transversal to $H=G_1$ in $G$,
\item[(c)] $G=\langle R\rangle$,
\item[(d)] $[R^{-1},R^{-1}]\le H$,
\item[(e)] $R^h = R$ for every $h\in H$?
\end{enumerate}
\end{problem}

Indeed, assume that all conditions of Problem \ref{Pr:CG} are satisfied. By (b), we can assume that $R=\{r_i;\;i\in Q\}$, where $1r_i=ir_1=i$ for every $i\in Q$. To show that the groupoid $(Q,*)$ defined by $i*j=ir_j$ is a loop, we use the following result, which can be deduced from \cite[Lemmas 2.1 and 2.2]{NiKe}:

\begin{lemma}\label{Lm:NiKeDual}
Let $H$ be a subgroup of $G$ and let $A$, $B$ be right transversals to $H$ such that $[A^{-1},B^{-1}]\le H$. Then both $A$ and $B$ are right transversals to every conjugate of $H$ in $G$.
\end{lemma}
\begin{proof}
Fix $x\in G$. Then $x=hb$ for unique $h\in H$, $b\in B$. Let $y\in G$. Then $yb^{-1} = ka$ for unique $k\in H$, $a\in A$. Then $y=kab = k[a^{-1},b^{-1}]ba = k[a^{-1},b^{-1}]h^{-1}xa\in Hxa$, so $G = \bigcup_{a\in A} Hxa$. Suppose that $Hxa\cap Hxc\ne \emptyset$ for some $a$, $c\in A$. Then $xac^{-1}x^{-1}\in H$, and $ac^{-1} = [a^{-1},b^{-1}]h^{-1}hbac^{-1}b^{-1}h^{-1}h[b^{-1},c^{-1}] = [a^{-1},b^{-1}]h^{-1}(xac^{-1}x^{-1})h[b^{-1},c^{-1}]\in H$, so $a=c$. This means that $xA$ is a right transversal to $H$. Then for a given $g\in G$ there are unique $h\in H$, $a\in A$ such that $xg=hxa$, $g=h^xa$. Hence $A$ is a right transversal to $H^x$. Similarly for $B$.
\end{proof}

By Lemma \ref{Lm:NiKeDual}, (b) and (d), $R$ is a right transversal to every conjugate of $H$ in $G$. By Lemma \ref{Lm:SameLoops}, $(Q,*)$ is a loop. Since $\rmlt{Q}=\langle R\rangle =G$ by (c) and $G$ is primitive by (a), $\mlt{Q}$ is primitive on $Q$ and hence $Q$ is simple by Proposition \ref{Pr:Simple}. By (d), $1[r_i^{-1},r_j^{-1}] = 1r_ir_jr_i^{-1}r_j^{-1}=1$ for every $i$, $j$, so $Q$ is commutative, $\mlt{Q}=\rmlt{Q}=G$, and $\inn{Q}=G_1=H$. By (e), for every $i\in Q$ and every $h\in \inn{Q}=H$ there is $j\in Q$ such that $r_i^h=r_j$. In this situation we have $1r_i^h=1r_j$, $ih = j$, and so $r_i^h = r_{ih}$. By Corollary \ref{Cr:ALoop}, $Q$ is automorphic. Since $|Q|=2^n>2$ by (a), $Q$ is not associative.

Conversely, if $Q$ is a non-associative finite simple commutative automorphic loop, then we can take $G=\mlt{Q}$, $R=\{R_i;\;i\in Q\}$, $H=G_1=\inn{Q}$, and observe (a) by Proposition \ref{Pr:Simple} and \cite[Proposition 6.1 and Theorem 6.2]{JKV1}, (b) by \S \ref{Sc:Folders}, (d) because $Q$ is commutative, (c) since $G = \mlt{Q}=\rmlt{Q}=\langle R\rangle$, and (e) by Lemma \ref{Lm:ACharacterization}.

\medskip

While attempting to answer Problem \ref{Pr:CG}, it will be useful to consider nontrivial consequences of (a)--(e). For instance,
\begin{enumerate}
\item[(f)] $\{r^2;\;r\in R\}\subseteq H$
\end{enumerate}
holds if and only if the loop $Q$ has exponent two, which must be true by \cite{JKV1}.

We are therefore interested in structural descriptions of primitive permutation groups of degree $2^n$. The following result of Guralnick and Saxl is from \cite{GuSa}:

\begin{theorem}[Guralnick and Saxl]\label{Th:GuSa}
Let $G$ be a primitive permutation group of degree $2^n$. Then either $G$ is of affine type, or $G$ has a unique minimal normal subgroup $N=S\times \cdots \times S = S^t$, $t\ge 1$, $S$ is a nonabelian simple group, and one of the following holds:
\begin{enumerate}
\item[(i)] $S=A_m$, $m=2^e\ge 8$, $n=te$, and the point stabilizer in $N$ is $N_1 = A_{m-1}\times \cdots \times A_{m-1}$, or
\item[(ii)] $S=PSL(d,q)$, $2^e=(q^d-1)/(q-1)\ge 8$, $d\ge 2$ is even, $q$ is odd, $m=te$, and the point stabilizer in $N$ is the direct product of maximal parabolic subgroups stabilizing either a $1$-space or hyperplane in each copy.
\end{enumerate}
\end{theorem}

\section*{Acknowledgement}

We thank Robert Guralnick for useful comments and for bringing Theorem \ref{Th:GuSa} to our attention. We thank the anonymous referee who suggested several improvements to the presentation of the paper.

\end{document}